\documentclass[12pt,a4paper]{amsart} 
\usepackage[T1,T2A]{fontenc}
\usepackage[utf8]{inputenc}
\usepackage[english]{babel}

\usepackage{amsmath}
\usepackage{amsfonts}
\usepackage{amssymb}
\usepackage{mathrsfs}
\usepackage{color}

\usepackage{amsthm}

\newtheorem{lemma}{Lemma} [section]

\newtheorem{theorem}{Theorem} [section]

\newtheorem*{remark}{Remark}

\newtheorem{result}{Result}[section]
\newtheorem*{conjecture}{Conjecture}
\usepackage[unicode,psdextra]{hyperref}
\newcommand\pdfmath[1]{\texorpdfstring{$#1$}{#1}}

\usepackage{graphics}
\usepackage{epstopdf}
\usepackage{epsfig}

\usepackage{url}
\usepackage{hyperref}

\usepackage[a4paper,top=2.5cm,bottom=2.8cm,
left=3.2cm,right=3.2cm]{geometry}
\markleft{\hfill \textsc{Polynomially growing integer sequences all whose terms are composite} \hfill}
\begin{document}
\bigskip
\bigskip

\title{Polynomially growing integer sequences all whose terms are composite}

\author{Dan Ismailescu}
\address{Mathematics Department, Hofstra University, Hempstead, NY 11549.}
\email{dan.p.ismailescu@hofstra.edu}

\author{Yunkyu James Lee}
\address{Columbia University, New York, NY 10027}
\email{yl4650@columbia.edu}

\begin{center}
	{\Large \textbf{Polynomially growing integer sequences all whose terms are composite}} \\
	\medskip
	\bigskip
	\textsc{Dan Ismailescu$^a$ \footnote{Corresponding author} and Yunkyu James Lee$^b$} \\

    $^a$ Mathematics Department, Hofstra University, Hempstead, NY 11549, USA\\
    email: \href{mailto:dan.p.ismailescu@hofstra.edu}{dan.p.ismailescu@hofstra.edu} \\
	
	$^b$
	Columbia University, New York, NY 10027, USA. \\
    e-mail: \href{mailto:yl4650@columbia.edu}{yl4650@columbia.edu} \\
	
\end{center}
\bigskip

\textbf{Abstract.} We identify pairs of positive integers $(t, d)$ with the property that the integer sequence with general term $\lfloor{n^t/d\rfloor}$ contains at most finitely many primes.



\bigskip

\section{Introduction}

Consider the integer sequence defined by the general term
\begin{equation*}
a_n= \left\lfloor\frac{n^2}{5}\right\rfloor, \quad n\ge 1.
\end{equation*}
This is sequence A118015 in the On-Line Encyclopaedia of Integer Sequences \cite{oeis}.

\noindent We list below the first few terms of this sequence.
\begin{align*}
&0, 0, 1, {\bf 3, 5, 7}, 9, 12, 16, 20, 24, 28, 33, 39, 45, 51, 57, 64, 72, 80, 88, 96, 105,\\
&115, 125, 135, 145, 156, 168, 180, 192, 204, 217, 231, 245, 259, 273, 288, 304, \\
&320, 336, 352, 369, 387, 405, 423, 441, 460, 480, 500, 520, 540, 561, 583, \ldots
\end{align*}
One can easily see that $a_4=3, a_5=5$, and $a_6=7$, but a quick computer check reveals that there are no other primes of the form $a_n$ for $n\le 10000$.
At first, we found this to be somewhat surprising, as there does not seem to be an obvious reason for this behavior. However,
it is not difficult to prove that $a_n$ is composite for all $n\ge 7$.
\begin{result}\label{result1}
{For all $n\ge 7$, $\left\lfloor n^2/5 \right\rfloor$ is composite.}
\end{result}
\begin{proof}
Consider each case $n \bmod 5$ separately.
\begin{align*}
&\text{If}\,\, n=5q+0 \,\,\text{then}\,\,\lfloor n^2/5\rfloor = 5q^2.\\
&\text{If}\,\, n=5q+1 \,\,\text{then}\,\,\lfloor n^2/5\rfloor = 5q^2+2q=q(5q+2).\\
&\text{If}\,\, n=5q+2 \,\,\text{then}\,\,\lfloor n^2/5\rfloor = 5q^2+4q=q(5q+4).\\
&\text{If}\,\, n=5q+3 \,\,\text{then}\,\,\lfloor n^2/5\rfloor = 5q^2+6q+1=(q+1)(5q+1).\\
&\text{If}\,\, n=5q+4 \,\,\text{then}\,\,\lfloor n^2/5\rfloor = 5q^2+8q+3=(q+1)(5q+3).\\
\end{align*} 
\end{proof}
We say that a sequence $\{a_n\}_{n\ge 1}$ with positive integer terms is \emph{eventually prime-free}, if there exists an index $n_0$ such that $a_n$ is composite for all $n\ge n_0$. Hence, the sequence with general term $\lfloor n^2/5\rfloor$ has this property.

The motivating reason behind this study is a famous conjecture of Bouniakowsky.
\begin{conjecture}\cite{bouniakowsky}
Let $f(n)$ be a polynomial with integer coefficients satisfying the following conditions

\emph{(a)} The leading coefficient of $f$ is positive.

\emph{(b)} $f$ is irreducible over the integers.

\emph{(c)} $\gcd(f(1),f(2),f(3),\ldots)=1$.

Then $f(n)$ is a prime for infinitely many values of $n$.
\end{conjecture}

We refer to polynomials satisfying conditions \emph{(a), (b), (c)} above as \emph{Bouniakowsky polynomials}.

The first condition is needed since we want to ensure that $f(n)>0$ for all sufficiently large $n$. The second condition is necessary because if $f(n)=g(n)h(n)$ for some integer coefficients polynomials $g$ and $h$, then there are only finitely many values of $n$ for which either $g(n)\in \{-1, 0, 1\}$ or $h(n)\in \{-1, 0, 1\}$, so $f(n)$ is necessarily composite for all sufficiently large $n$.

Finally, while it is necessary that the coefficients of $f$ do not share a common factor $>1$, the third condition is actually stronger. Indeed, $f(n)=n^2 +n+2$ satisfies the first two conditions and its coefficients are relatively prime; however $f(n)$ is even for all integers $n$, and therefore it is prime only finitely many times.

Despite its venerable age, Bouniakowsky's conjecture is still wide open; no single case of the conjecture for polynomials of degree greater than $1$ is proved. In particular, it is unknown whether $n^2+1$ is prime infinitely often, although numerical evidence is consistent with the conjecture.
Worse yet, it is unknown if a general Bouniakowsky polynomial will always produce at least one prime number. For example, $n^{12}+488669$ produces no primes until $n\in \{616980, 764400, 933660,\ldots\}$; this is sequence A122131 in \cite{oeis}.

To date, the only case of Bouniakowsky's conjecture that can be rigorously proved is that of linear polynomials. This is the famous Dirichlet theorem \cite{dirichlet}, which states that every arithmetic sequence with coprime first term and common difference contains infinitely many primes.

Given the current state of affairs in regards to Bouniakowsky's conjecture, it is not at all surprising that there are no known examples of integer sequences which

$(i)$ grow at a polynomial rate, 

$(ii)$ are defined by a natural expression, and 

$(iii)$ are eventually prime-free in a manner that is not immediately obvious.

In this paper, we study sequences whose general term is of the form $\lfloor n^t/d\rfloor$ where $t$ and $d$ are positive integers both greater than $1$.
Clearly, such sequences satisfy both requirements $(i)$ and $(ii)$. It remains to investigate under what circumstances condition \emph{(iii)} is also fulfilled.

\section{Case \pdfmath{t=2}}

\begin{theorem}\label{thm2}
For every $d\in \{2,3,4,5,8,12,16\}$ there are at most finitely many primes of the form $\lfloor n^2/d\rfloor$.
\end{theorem}

\begin{proof}
Let us note that for every positive integer $n$, we have
\begin{align*}
&n^2\equiv 0,1\!\!\!\!\! \pmod{2}, \,\,  n^2\equiv 0,1\!\!\!\!\! \pmod 3,\,\, n^2\equiv 0,1\!\!\!\!\! \pmod 4,\\
&n^2\equiv 0,1,4\!\!\!\!\! \pmod{5}, \,\,  n^2\equiv 0,1,4\!\!\!\!\! \pmod 8\\
&n^2\equiv 0,1,4,9\!\!\!\!\! \pmod{12}, \,\,  n^2\equiv 0,1,4,9\!\!\!\!\! \pmod {16}.
\end{align*}

Hence, for every $n\ge 1$ and every $d$ in $\{2,3,4,5,8,12,16\}$ we have
\begin{equation*}
n^2\equiv r^2 \!\!\!\!\! \pmod{d}\,\,\text{where}\,\, 0\le r^2 <d.
\end{equation*}

It follows that for every $n\ge 1$ and every $d$ in $\{2,3,4,5,8,12,16\}$
\begin{equation*}
n^2=Md+r^2 \,\,\text{for some integer}\,\, M\ge 0 \,\, \text{and}\,\, 0\le r^2 < d.
\end{equation*}
But then
\begin{equation*}
\left\lfloor\frac{n^2}{d}\right\rfloor= \left\lfloor M+\frac{r^2}{d}\right\rfloor=M=\frac{n^2-r^2}{d}=\frac{(n-r)(n+r)}{d}.
\end{equation*}
The last quantity in the above equality is a composite integer for all $n\ge 2d$, as for such $n$ we have $n-r>d$ and $n+r>d$.
Hence, for each $d \in \{ 2,3,4,5,8,12,16\}$, the number $\lfloor n^2/d\rfloor$ is prime for at most finitely many values of $n$, as claimed.

\end{proof}

\section{Case \pdfmath{t=3}}

\begin{theorem}\label{thm3}
For every $d\in \{2,9\}$ there are at most finitely many primes of the form $\lfloor n^3/d\rfloor$.
\end{theorem}

\begin{proof}

For every $n\ge 1$, we have $n^3\equiv 0, 1 \pmod 2$ and $n^3\equiv 0,1,8 \pmod 9$.
Hence, $n^3=Md+r^3$ for some integers $M\ge 0$ and $0\le r^3 < d$.
But then
\begin{equation*}
\left\lfloor\frac{n^3}{d}\right\rfloor= \left\lfloor M+\frac{r^3}{d}\right\rfloor=M=\frac{n^3-r^3}{d}=\frac{(n-r)(n^2+nr+r^2)}{d}.
\end{equation*}

The last quantity in the above equality is a composite integer for all $n\ge 2d$, as for such $n$ we have $n-r>d$ and $n^2+nr+r^2>d$.
Hence, for each $d \in \{ 2, 9\}$, the number $\lfloor n^3/d\rfloor$ is prime for at most finitely many values of $n$, as claimed.
\end{proof}

Before investigating the problem for values $t\ge 4$ several observations are in order.
The first two statements are rather trivial but we record them for future reference.
\begin{remark}\label{obs1}
If $s$ and $t$ are positive integers with $s\,|\, t$ and $\lfloor n^{s}/d\rfloor$ is prime for at most finitely many values of $n$, then the same is true for $\lfloor n^{t}/d\rfloor$.
\end{remark}
This is obvious since $\{\lfloor n^{t}/d\rfloor\}$ is a subsequence of $\{\lfloor n^{s}/d\rfloor\}$.

Next, we generalize the argument used in the proofs of both Theorem \ref{thm2} and Theorem \ref{thm3}.
\begin{lemma}\label{lemma1}
Given integers $t\ge 2$ and $d\ge 2$, let $n\ge 2d$ be an integer with the property that $n^t\equiv r^t \pmod d$ such that $0\le r^t<d$. Then $\lfloor n^t/d\rfloor$ is composite.
\end{lemma}
\begin{proof}

Since $n\ge 2d$, there exists a positive integer $M$ such that $n=Md+r^t$. Then
\begin{equation*}
\left\lfloor\frac{n^t}{d}\right\rfloor= \left\lfloor M+\frac{r^t}{d}\right\rfloor=M=\frac{n^t-r^t}{d}=\frac{(n-r)\sum_{k=0}^{t-1} n^kr^{t-1-k}}{d}.
\end{equation*}

It follows that both factors of the numerator of the last quantity above are greater than $d$. The conclusion follows.
\end{proof}

Finally, note that the set of values of $d$ for which we can prove that $\{\lfloor n^{3}/d\rfloor\}$ is eventually prime-free is considerably smaller than the corresponding set of values of $d$ for which $\{\lfloor n^{2}/d\rfloor\}$ has the same property. We believe this is not accidental; in fact, we venture the following
\begin{conjecture}\label{conjodd}
Let $t$ and $d$ be positive integers, $t$ odd,  $d\ge 2$.

\noindent(a) If $t\equiv 0 \pmod 3$ and $d\neq 2, 9$ then there are infinitely many primes of the form $\lfloor n^t/d\rfloor$.

\noindent(b) If $t\not\equiv 0 \pmod 3$ and $d\neq 2$ then there are infinitely many primes of the form $\lfloor n^t/d\rfloor$.
\end{conjecture}

\section{Case \pdfmath{t=4}}

\begin{theorem}\label{thm4}
For every $d\in \{2,3,4,5,8,12,16\} \cup \{24,40\}$ the sequence with general term $\lfloor n^4/d\rfloor$ is eventually prime-free.
\end{theorem}

\begin{proof}
From Theorem \ref{thm2} we already know that $\{\lfloor{n^2/d}\rfloor\}$ is eventually prime-free if $d$ belongs to the first set of the union above.
Since $\lfloor{n^4/d}\rfloor$ is a subsequence of $\lfloor{n^2/d}\rfloor$ there is nothing left to prove.

The only new cases are $d=24$ and $d=40$.

But it can be easily checked that for every $n\ge 1$, we have
\begin{equation*}
n^4\equiv 0, 1, 9, 16 \!\!\!\!\! \pmod {24} \quad \text{and}\quad  n^4\equiv 0,1,16,25 \!\!\!\!\! \pmod {40}.
\end{equation*}

So, for any $n\ge 2d$ we have that $n^4=Md+r^2$ where $M\ge 2$ and $0\le r^2 <d$.

One can now apply the argument from Lemma \ref{lemma1} to obtain
\begin{equation*}
\left\lfloor\frac{n^4}{d}\right\rfloor= \left\lfloor M+\frac{r^2}{d}\right\rfloor=M=\frac{n^4-r^2}{d}=\frac{(n^2-r)(n^2+r)}{d}.
\end{equation*}
Again, this is composite for all $n\ge 2d$. The proof is complete.
\end{proof}


\section{Case \pdfmath{t=6}}

\begin{theorem}\label{thm6}
For every $d\in \{2,3,4,5,8,9,12,16\} \cup\{7,24,56,72\}$ the sequence  $\{\lfloor n^6/d\rfloor\}_n$ is eventually prime-free.
\end{theorem}

\begin{proof}

Since $\{\lfloor n^6/d\rfloor\}_n$ is a subsequence of both $\{\lfloor n^2/d\rfloor\}_n$ and $\{\lfloor n^3/d\rfloor\}_n$, the first set of values of $d$ is covered by Theorem \ref{thm2} and by Theorem \ref{thm3}.

The cases $d \in \{ 7,24,56,72 \}$ can also be easily handled as it can be checked that for every integer $n$, we have
$n^6\equiv 0, 1 \pmod {7}$, $n^6\equiv 0,1,16,25 \pmod {24}$,
$n^6\equiv 0, 1, 8, 49 \pmod {56}$, and  $n^6\equiv 0,1,9,64 \pmod {72}$.

Hence, for any $n$ we either have $n^6\equiv r^2 \pmod d$ or $n^6 \equiv r^3 \pmod d$ with $0\le r^2<d$ or $0\le r^3 <d$, respectively.

Then, as in the proofs of Theorem \ref{thm3} and Theorem \ref{thm4}, we have
\begin{align*}
&\left\lfloor\frac{n^6}{d}\right\rfloor= \frac{n^6-r^2}{d}=\frac{(n^3-r)(n^3+r)}{d}\,\, \text{or} \\ &\left\lfloor\frac{n^6}{d}\right\rfloor= \frac{n^6-r^3}{d}=\frac{(n^2-r)(n^4+n^2r+r^2)}{d}.
\end{align*}
Either way, $\left\lfloor n^6/d\right\rfloor$ is certainly composite if $n\ge 2d$.

\end{proof}


\section{Case \pdfmath{t=8}}

\begin{theorem}\label{thm8}
For every $d \in \{2,3,4,5,8,12,16,24,40\} \cup \{17, 32, 48, 80, 112\}$ the sequence $\{\lfloor n^8/d\rfloor\}$ is eventually prime-free.
\end{theorem}

\begin{proof}
Since  $\{\lfloor n^8/d\rfloor\}_n$ is a subsequence of $\{\lfloor n^4/d\rfloor\}_n$, the first set of values of $d$ is covered by Theorem \ref{thm4}.

The cases $d=17$ and $d=32$ are also easy since for every integer $n$ we have
$n^8\equiv 0, 1, 16 \pmod {17}$ and $n^8\equiv 0, 1 \pmod {32}$. As all these remainders are perfect squares, the same argument as in Theorem \ref{thm4} can be applied.

However, for each $d\in \{48, 80, 112\}$ there exist values of $n$ such that $n^8 \pmod d$ is not a perfect square. Indeed, one can see that
\begin{align*}
&n^8\equiv 0,1,16,{\bf 33} \pmod {48}, \\
&n^8\equiv 0,1,16, {\bf 65} \pmod {80},\,\, \text{and}\\
&n^8\equiv 0,1,16, {\bf 32},49,64, {\bf 65}, 81 \pmod {112}.
\end{align*}

Nevertheless, we claim that Theorem \ref{thm8} holds even for these values of $d$.
We consider each case separately.

Suppose first that $d=48$.
Then, if $n\not\equiv 3 \pmod 6$ we have that $n^8\equiv 0, 1, 16 \pmod {48}$ and therefore $\lfloor n^8/48\rfloor$ is composite for all sufficiently large $n$.

It remains to deal with the case $n=6m+3$. Expanding $n^8/48$ we obtain

\begin{equation*}
\frac{n^8}{48}= \sum_{k=0} ^8 {8 \choose {k}} 2^{4-k}3^7m^{8-k}=\sum_{k=0}^6 {8 \choose {k}}. 2^{4-k}3^7m^{8-k}+3^7m(7m+1)+\frac{3^7}{16}.
\end{equation*}

One can readily check that every coefficient of the first sum is an even integer, and so is the term $3^7m(7m+1)$. Finally, $\lfloor 3^7/16\rfloor =136$, hence $\lfloor n^8/48 \rfloor$ is necessarily even when $n\equiv 3\pmod 6$, and therefore composite.

We encounter a similar situation when $d=80$.
Then, if $n\not\equiv 5 \pmod {10}$ we have that $n^8\equiv 0, 1, 16 \pmod {80}$ and therefore $\lfloor n^8/80\rfloor$ is composite for all sufficiently large $n$.

It remains to deal with the case $n=10m+5$. Expanding $n^8/80$ we obtain
\begin{equation*}
\frac{n^8}{80}= \sum_{k=0}^8 {{8} \choose {k}}2^{4-k}5^7m^{8-k}=\sum_{k=0}^6 {{8} \choose {k}}2^{4-k}5^7m^{8-k}+5^7m(7m+1)+\frac{5^7}{16}
\end{equation*}
We use again the fact that every coefficient of the first sum is an even integer, and so is the term $5^7m(7m+1)$. Finally, $\lfloor 5^7/16\rfloor =4882$, hence $\lfloor n^8/80 \rfloor$ is necessarily even when $n\equiv 5\pmod {10}$, and therefore composite.

Finally, suppose that $d=112$.

If $n\not\equiv \pm2, \pm3 \pmod {14}$ then $n^8\equiv 0, 1, 16, 49, 64, 81 \pmod {112}$ and therefore $\lfloor n^8/112\rfloor$ is composite for all sufficiently large $n$.

If $n\equiv\pm 2 \pmod{14}$ then $n^8 \equiv 32 \pmod {112}$ and since $32$ is not a perfect square we need a different argument.
Assume that $n=14m+2\epsilon$ where $\epsilon=\pm 1$. Expanding we get
\begin{equation*}
\frac{n^8}{112}= \frac{16}{7}(7m+\epsilon)^8=16\sum_{k=0}^7 {{8} \choose {k}}7^{7-k}m^{8-k}\epsilon^k+\frac{16}{7}.
\end{equation*}
The first sum is an even integer; since $\lfloor 16/7\rfloor =2$ it follows that $\lfloor n^8/112\rfloor$ is even for all $n\equiv \pm2 \pmod {14}$
and therefore composite.

A very similar argument can be used if $n\equiv\pm 3 \pmod{14}$;  in this case then $n^8 \equiv 65 \pmod {112}$.
Assume that $n=14m+3\epsilon$ where $\epsilon=\pm 1$. Expanding we get
\begin{equation*}
\frac{n^8}{112}= (14m+3\epsilon)^8=\sum_{k=0}^5 {{8} \choose {k}}2^{4-k}3^k7^{7-k}m^{8-k}\epsilon^k+ 3^6m\epsilon(49m+3\epsilon)+\frac{3^8}{112}.
\end{equation*}
All coefficients of the first sum are even integers, and so is the second term. Also, $\lfloor 3^8/112\rfloor=58\equiv 0 \pmod 2$, so $\lfloor n^8/112\rfloor$ is even for all $n\equiv \pm3 \pmod {14}$. This completes the proof.
\end{proof}

\section{Case \pdfmath{d=24}}

Until now, for a given even integer $t$ we tried to discover the values of $d$ which make the sequence $\{\lfloor n^t/d\rfloor\}_n$ eventually prime-free.
We pause this search for now in order to present a value of $d$ which works for all even $t\ge 4$.

\begin{theorem}\label{thmd24}
For every even $t \ge 4$ the sequence $\{\lfloor n^t/24\rfloor\}$ is eventually prime-free.
\end{theorem}

\begin{proof}
Suppose that $t=2s$ where $s\ge 2$.
We distinguish a few cases depending on the remainder of $n$ modulo $6$.

If $n\equiv \pm 1 \pmod 6$ then $n^2\equiv 1 \pmod{24}$, from which $n^{2s}\equiv 1 \pmod{24}$ for all $s\ge 1$. But this implies that
\begin{equation*}
\left\lfloor\frac{n^t}{24}\right\rfloor= \frac{n^{2s}-1}{24}=\frac{(n^s-1)(n^s+1)}{24}
\end{equation*}
and this is certainly composite for all large enough $n$.

If $n\equiv \pm 2 \pmod 6$ then $n^2\equiv 4, 16 \pmod{24}$, from which $n^{2s}\equiv 16 \pmod{24}$ for all $s\ge 2$. But this implies that
\begin{equation*}
\left\lfloor\frac{n^t}{24}\right\rfloor= \frac{n^{2s}-16}{24}=\frac{(n^s-4)(n^s+4)}{24}
\end{equation*}
and again this is composite for all sufficiently large $n$.

If $n\equiv 3 \pmod 6$ then $n^2\equiv 9 \pmod{24}$, from which $n^{2s}\equiv 9 \pmod{24}$ for all $s\ge 1$. This implies that
\begin{equation*}
\left\lfloor\frac{n^t}{24}\right\rfloor= \frac{n^{2s}-9}{24}=\frac{(n^s-3)(n^s+3)}{24}.
\end{equation*}
The same conclusion as in the previous cases follows.

Finally, assume that $n\equiv 0 \pmod 6$. Then $n=6m$ for some positive integer $n$. It follows that
\begin{equation*}
\left\lfloor\frac{n^t}{24}\right\rfloor= \frac{2^t3^tm^t}{2^3\cdot 3}=2^{t-3}3^{t-1}m^t,
\end{equation*}
and this is certainly a composite integer since $t\ge 4$.
\end{proof}
\section{Case \pdfmath{t=10}}

\begin{theorem}\label{thm10}
For every $d \in \{2,3,4,5,8,12,16\} \cup \{11,24,40\}$ the sequence with general term $\lfloor n^{10}/d\rfloor$ is eventually prime-free.
\end{theorem}

\begin{proof}
The values of $d$ in the first set are covered by Theorem \ref{thm2} since  $\{\lfloor n^{10}/d\rfloor\}_n$ is a subsequence of $\{\lfloor n^2/d\rfloor\}_n$.
The case $d=24$ is a consequence of Theorem \ref{thmd24}.

The case $d=11$ is an immediate consequence of Fermat's Little Theorem since in this case $n^{10}\equiv 0, 1 \pmod{11}$ for all integers $n$.

Finally, if $d=40$ and $n\not\equiv \pm 2 \pmod {10}$ then $n^{10}\equiv 0, 1, 9, 16, 25 \pmod{40}$ and the conclusion easily follows.
It remains to look at the situation when $n\equiv \pm 2 \pmod{10}$, as in this case $n^{10}\equiv 24 \pmod{40}$.

Assume that $n=10m+2\epsilon$ where $\epsilon=\pm 1$. Expanding we obtain
\begin{equation*}
\frac{n^{10}}{40}= \frac{(10m+2\epsilon)^{10}}{40}=\sum_{k=0}^9 {{10} \choose {k}}2^7 5^{9-k} m^{10-k}\epsilon^k+ \frac{2^{10}}{40}.
\end{equation*}
It is immediate that all coefficients of the first sum are integers multiples of $5$, while $\lfloor 2^{10}/40\rfloor=25\equiv 0 \pmod 5$, so $\lfloor n^{10}/40\rfloor$ is an integer multiple of $5$ for all $n\equiv \pm2 \pmod {10}$. This completes the proof.

\end{proof}

\section{Cases \pdfmath{12\le t \le 54}}

One can continue the search for pairs $(t,d)$ for which the sequence with general term $\lfloor n^{s}/d\rfloor$ is eventually prime-free
for larger values of $t$.
In Table \ref{table1} we tabulated our findings for the entire range $2\le t\le 54$. In boldface we present what \emph{we believe} to be the primitive pairs $(t,d)$ with the desired property.

We say a pair $(t,d)$ is \emph{primitive} if $\lfloor n^{t}/d\rfloor$ is eventually prime-free but there is no proper divisor $s$ of $t$ for which $\lfloor n^{s}/d\rfloor$ has the same property.

For instance, we are inclined to think that the pair $(t,d)=(6,7)$ is primitive because we can easily show (via Fermat's Little Theorem) that there are only finitely many primes of the form $\lfloor n^6/7\rfloor$. However, one still has to show that there are infinitely many primes of the form $\lfloor n^2/7\rfloor$ as well as infinitely many primes of the form $\lfloor n^3/7\rfloor$.

On one hand, disproving such a statement is equivalent to finding a counterexample to Bouniakowsky's conjecture. On the other hand, a proof seems to be equally challenging as we are not aware of a single instance of a polynomial of degree $\ge 2$ for which the conclusion of Bouniakowsky's conjecture is satisfied.

In conclusion, deciding whether a pair $(t,d)$ is primitive or not is currently beyond anyone's abilities.

In light of Conjecture \ref{conjodd}, we do not list any odd values of $t$ other than $t=3$. Also, some even values of $t$ (such as $t=14, 26, 34, 38, 44, 46$) are also missing from Table \ref{table1}; this is because we were unable to identify any primitive pairs for these particular values of $t$.

Proving that each pair appearing in Table \ref{table1} has indeed the desired property is relatively straightforward; all proofs follow the model presented in the previous sections. Broadly speaking, for a given pair $(t,d)$ and for a given congruence class $n \pmod d$, this subsequence of $\{n^t/d\}_n$ behaves in one of two ways:
\begin{itemize}
\item{The remainder of $n^t$ modulo $d$ is a perfect power whose exponent shares a nontrivial common factor with $t$ itself (in most cases the remainder is a perfect square). In this case, $d\lfloor n^t/d\rfloor$ can be written as a difference of like powers, that is, as the product of two polynomials in $n$; as a result, $\lfloor n^t/d\rfloor$ is composite for all sufficiently large $n$.}
\item{The remainder of $n^t$ modulo $d$ is not perfect power, case in which we try to prove that all numbers of the form $\lfloor n^t/d\rfloor$ (for $n$ belonging to that particular congruence class modulo $d$) are divisible by a common prime number. The proof always reduces to checking only finitely many values of $n$.}
\end{itemize}

We illustrate this process by proving the following
\begin{theorem}
There are only finitely many primes of the form $\lfloor n^{30}/1116\rfloor$.
\end{theorem}

\begin{proof}
Note that $1116=2^2\cdot3^2\cdot 31$.
We divide the proof into two cases.

If $n\equiv 0 \pmod 6$ then $n^{30}\equiv 0\pmod {1116}$ or $n^{30}\equiv 900\pmod {1116}$ depending on whether $n$ is an integer multiple of $31$ or not.
In either case, the remainder is a perfect square and therefore $\lfloor n^{30}/1116\rfloor$ is composite for all but finitely many such $n$.

Consider next that $n\not\equiv 0 \pmod 6$. Write $n=186m+r=2\cdot3\cdot31\cdot m+r$ where $m$ and $r$ are integers with
$0<r<186$ and $r\not\equiv 0 \pmod 6$.
We claim that in this case, we have that
\begin{equation}\label{e30}
\left\lfloor \frac{n^{30}}{1116}\right\rfloor \equiv 5m\,r^{28}(2697m+r)+\left\lfloor \frac{r^{30}}{1116}\right\rfloor \pmod 6
\end{equation}
Indeed, by expanding $n^{30}/1116$ we obtain
\begin{equation*}
\frac{n^{30}}{1116}=\frac{(186m+r)^{30}}{2^2\cdot3^2\cdot31}=\sum_{k=0}^{30}{{30} \choose k}2^{28-k}3^{28-k}31^{29-k}m^{30-k}r^k
\end{equation*}
Obviously, for $0\le k\le 27$ the general term of the above sum is an integer multiple of $6$ while the last three terms are
\begin{equation*}
{{30} \choose 28}31q^2r^{28}=5\cdot2697m^2r^{28},\,\, {{30} \choose 29}2^{-1}3^{-1}mr^{29}= 5qr^{29},\,\,\text{and}\,\, r^{30}/1116,
\end{equation*}
respectively. Equality (\ref{e30}) follows after taking integer parts on both sides.

Next, we claim that $\lfloor n^{30}/1116\rfloor$ is an even integer if $r\equiv \pm 1, \pm 2 \pmod{6}$, and that $\lfloor n^{30}/1116\rfloor$ is an integer multiple of $3$ if $r\equiv 3 \pmod{6}$.

Indeed, one can easily see that $5m\,r^{28}(2697m+r)$ is even for any choice of $m$ and $r$, and it is an integer multiple of $3$ if $r\equiv 3 \pmod{6}$.
Moreover, a finite search of the range $0<r<186$ reveals that $\lfloor r^{30}/1116\rfloor$ is an even integer if $r\not\equiv 0 \pmod 3$ and an integer multiple of $3$ if $r\equiv 3\pmod 6$.
This completes the proof.

\end{proof}

\section{Three general families}

We present three simple instances of eventually prime-free sequences of the form $\{\lfloor n^t/d\rfloor\}_n$.

\begin{theorem}\label{p}
Let $p$ be an odd prime. Then, the sequence with general term $\lfloor n^{p-1}/p\rfloor$ is eventually prime-free.
\end{theorem}

\begin{proof}

This is a simple consequence of Fermat's Little Theorem, as for every positive integer $n$ we have that
$n^{p-1}\equiv 0, 1 \pmod{p}$. So, $n^{p-1}= Mp +r^2$ where $r\in\{0, 1\}$. But then,
\begin{equation*}
\left\lfloor\frac{n^{p-1}}{p}\right\rfloor= \left\lfloor M+\frac{r^2}{p}\right\rfloor=M=\frac{n^{p-1}-r^2}{p}=\frac{(n^{(p-1)/2}-r)(n^{(p-1)/2}+r)}{p},
\end{equation*}
which is composite for all sufficiently large $n$.
\end{proof}

\begin{theorem}\label{pp}

Let $p$ be an odd prime of the form $p=q^2+1$. Then, the sequence with the general term $\lfloor n^{(p-1)/2}/p\rfloor$ is eventually prime-free.
In particular, each of the following sequences has the desired property:
\begin{equation*}
\left\lfloor \frac{n^2}{5}\right\rfloor, \left\lfloor \frac{n^8}{17}\right\rfloor, \left\lfloor \frac{n^{18}}{37}\right\rfloor, \left\lfloor \frac{n^{50}}{101}\right\rfloor,\left\lfloor \frac{n^{98}}{197}\right\rfloor, \left\lfloor \frac{n^{128}}{257}\right\rfloor, \ldots
\end{equation*}
\end{theorem}
\begin{proof}
Note that $q$ must be even and therefore $p\equiv 1\pmod 4$.

Denote, $N=n^{(p-1)/4}$. Then, Fermat's Little Theorem gives
\begin{equation*}
N^4\equiv 0, 1 \!\!\!\!\pmod{p}\!\iff\! N^2=0,1,-1\!\!\!\! \pmod{p} \!\iff\! N^2\equiv 0, 1, q^2\!\!\!\! \pmod{p}
\end{equation*}

Finally
\begin{equation*}
\left\lfloor\frac{n^{(p-1)/2}}{p}\right\rfloor=\left\lfloor\frac{N^2}{p}\right\rfloor= \frac{N^2-r^2}{p}=\frac{(N-r)(N+r)}{p},\, \text{where}\,\, r=0, 1, q.
\end{equation*}
This is composite for all sufficiently large $n$.
\end{proof}

The sequence of primes of the form $q^2+1$ appears as A002496 in \cite{oeis}. As mentioned earlier, it is unknown whether this sequence has infinitely many terms.
\begin{align*}
&2, 5, 17, 37, 101, 197, 257, 401, 577, 677, 1297, 1601, 2917, 3137, 4357, \\
&5477, 7057, 8101, 8837, 12101, 13457, 14401, 15377, 15877, 16901, 17957, \\
&21317, 22501, 24337, 25601, 28901, 30977, \ldots
\end{align*}

\begin{theorem}\label{ppp}
Let $p\ge 5$ be an odd prime of the form $p=q^6+q^3+1$. Then, the sequence with general term $\lfloor n^{(p-1)/3}/p\rfloor$ is eventually prime-free.In particular, each of the following sequences has the desired property:
\begin{equation*}
\left\lfloor \frac{n^{24}}{73}\right\rfloor, \left\lfloor \frac{n^{252}}{757}\right\rfloor, \left\lfloor \frac{n^{87552}}{262657}\right\rfloor, \ldots
\end{equation*}
\end{theorem}

\begin{proof}
It is easy to check that necessarily $p\equiv 1\pmod 9$.

Denote, $N=n^{(p-1)/9}$. Then, by Fermat's Little Theorem, we obtain
\begin{equation*}
N^9\equiv 0, 1 \!\!\!\!\pmod{p}\!\iff\! N^3=0,1,q^3,q^6\!\!\!\! \pmod{p}
\end{equation*}
The above equivalence is a result of Lagrange's theorem which guarantees that a
cubic polynomial has at most three incongruent solutions modulo a given prime.
Finally, for some $r \in \{0, 1, q, q^2\}$ we have
\begin{equation*}
\left\lfloor\frac{n^{(p-1)/3}}{p}\right\rfloor=\left\lfloor\frac{N^3}{p}\right\rfloor= \frac{N^3-r^3}{p}=\frac{(N-r)(N^2+Nr+r^2)}{p},
\end{equation*}
which is composite for all sufficiently large $n$.
\end{proof}

The sequence of primes of the form $q^6+q^3+1$ appears as A162601 in \cite{oeis}. As mentioned earlier, it is unknown whether this sequence has infinitely many terms.
\begin{align*}
&3, 73, 757, 262657, 1772893, 64008001, 85775383, 308933353,\\
& 729027001, 15625125001, 17596420453, 30841155073, 46656216001, \\
& 225200075257, 885843322057, 1126163480473, \ldots
\end{align*}
\section{ A final result}

Based on the results above, one may expect that greater values of $d$ require larger values of $t$ for $\lfloor n^t/d\rfloor$ to be eventually prime-free.
However, we can prove the following result.
\begin{theorem}\label{thmlast}
Let $p$ be an odd prime and let $c$ be such that $-c-1$ is a quadratic nonresidue modulo $p$. For all integers $n\ge 1$ define
\begin{equation}\label{qn}
q_n=\left\lfloor\frac{(p-1)!(n^2+c)}{p}\right\rfloor.
\end{equation}
Then $q_n$ has a prime factor smaller than $p$ and therefore is composite for all $n\ge 1$.
\end{theorem}
 
\begin{proof}
Equation \eqref{qn} can be written as
\begin{equation*}
(p-1)!(n^2+c)=pq_n+r\, \quad \text{where} \,\, 0\le r\le p-1.
\end{equation*}
Note that $r\neq 1$. Indeed, if one assumes that $r=1$ then by using Wilson's theorem in equality \eqref{qn} it follows that
$-n^2-c\equiv 1\pmod{p}$, that is, $n^2\equiv -c-1 \pmod{p}$, which is impossible by the choice of $c$. 

Hence, either $r=0$ or $2\le r\le p-1$. 

In the first case, the equality $(p-1)!(n^2+c)=pq_n$ immediately implies $q_n\equiv 0 \pmod{(p-1)!}$ and therefore $q_n$ must be composite.. 

In the second case,  $2\le r\le p-1$ implies that $pq_n=(p-1)!(n^2+c)-r\equiv 0\pmod{r}$ since $(p-1)!\equiv 0 \pmod{r}$. It follows that $q_n\equiv 0\pmod r$ and we are done.
\end{proof}



\begin{table}[ht]
\begin{centering}
\begin{tabular}{c|c}
\hline
    $t$        & $d$             \\ 
\hline
$2$         & ${\bf 2, 3, 4, 5, 8, 12, 16}$  \\ 
$3$         & ${\bf 2, 9}$ \\ 
$4$         & $2,3,4,5,8,12,16, {\bf 24, 40}$   \\
$6$         & $2,3,4,5,{\bf 7},8,9,12,16, {\bf 24, 56, 72}$   \\
$8$         & $2,3,4,5,8,12,16,{\bf 17},24,{\bf 32}, 40,{\bf 48, 80, 112}$   \\
$10$        & $2,3,4,5,8,{\bf 11}, 12, 16,{\bf 24, 40}$   \\
$12$        & $2,3,4,5,8,9,12,{\bf 13},16,24,40,{\bf 112, 144, 240} $   \\
$16$        & $2,3,4,5,8,12,16,17,24,32, 40,48, {\bf 64}, 80, 112,{\bf 544} $   \\
$18$        & $2,3,4,5,7,8,9,12,16,{\bf 19}, 24, {\bf 27, 36, 37, 54}, 56,{\bf 63}, 72, {\bf 252}$   \\
$20$        & $2,3,4,5,8,12,16, 24,{\bf 25}, 40, {\bf 200}$  \\
$22$        & $2, 3, 4, 5, 8, 12, 16, {\bf 23}$  \\
$24$        & $2,3,4,5,8,9,12,13,16,24,40,{\bf 73},112,144,{\bf 208}, 240,{\bf 288, 576}$   \\
$28$        & $2, 3, 4, 5, 8, 12, 16, {\bf 29}$  \\
$30$        & $2,3,4,5,7,8,9,12,16, 24,{\bf 31}, 56, 72, {\bf 1116}$   \\
$32$        & $2,3,4,5,8,12,16,17,24,32, 40,48, 64, 80, 112, {\bf 128, 192}, 544$   \\
$36$        & $ 2,3,4,5,7,8,9,12,16,19, 24,27, 36, 37, 54, 56,63, 72, 252, {\bf 432, 2664}$   \\
$40$        & $2,3,4,5,8,12,16, 17,24,25,32, 40,48, 80, 112,{\bf 176}, 200, {\bf 800, 1968}$   \\
$42$        & $2,3,4,5,7,8,9,12,16, 24,{\bf 43, 49}, 56, 72$ \\
$50$        & $ 2,3,4,5,8,11,12,16,24,40, {\bf 101}$   \\
$52$        & $2,3,4,5,8,12,16,24, 40, {\bf 53}$ \\
$54$        & $2,3,4,5,7,8,9,12,16, 19, 24, 27, 36, 37, 54, 56, 63, 72,{\bf 81}, 252,{\bf 1404}$   \\
\hline
\end{tabular}
\medskip
\caption{Pairs $(t, d)$ for which the sequence with general term $\lfloor n^t/d\rfloor$ is provably eventually prime-free. In boldface are values of $d$ with the property that there is no proper divisor $s$ of $t$ for which we can prove that $\lfloor n^{s}/d\rfloor$ is eventually prime-free.}
\label{table1}
\end{centering}
\end{table}
\section{Conclusions}
There are many interesting questions left unanswered. On one hand, we are far from being certain that Table \ref{table1} is complete.
While the problem of characterizing the pairs $(t,d)$ for which $\lfloor n^t/d\rfloor$ is eventually prime-free is undoubtedly a very difficult one, one may ask for other results similar to those stated in Theorems \ref{p}, \ref{pp}, \ref{ppp}, and \ref{thmlast}.

\end{document}